\documentclass{amsart}
\usepackage{amsmath, amssymb, latexsym,  amsthm, amscd}

\newcommand{ \PP }{{\mathbb P}}
\newcommand{ \CC }{{\mathbb C}}

\newtheorem{thm}{Theorem}[section]
 \newtheorem{cor}[thm]{Corollary}
 \newtheorem{lem}[thm]{Lemma}
 \newtheorem{prop}[thm]{Proposition}
 
 \theoremstyle{definition}
 \newtheorem{defn}[thm]{Definition}
 \theoremstyle{remark}
 \newtheorem{rem}[thm]{Remark}

\begin{document}

\title{Refined methods for the identifiability of tensors}

\author[C. Bocci]{Cristiano Bocci}
\address[C. Bocci]{Dipartimento di Ingegneria dell'Informazione e Scienze Matematiche 
\\ Universit\`{a} di Siena \\ Pian dei Mantellini 44 \\ 53100 Siena,  Italia}
\email{cristiano.boccii@unisi.it}
\urladdr{\tt{http://www.mat.unisi.it/newsito/docente.php?id=16}}

\author[L. Chiantini]{Luca Chiantini}
\address[L. Chiantini]{Dipartimento di Ingegneria dell'Informazione e Scienze Matematiche 
\\ Universit\`{a} di Siena \\ Pian dei Mantellini 44 \\ 53100 Siena,  Italia}
\email{luca.chiantini@unisi.it}
\urladdr{\tt{http://www.mat.unisi.it/newsito/docente.php?id=4}}

\author[G. Ottaviani]{Giorgio Ottaviani}
\address[G. Ottaviani]{Dipartmento di Matematica e Informatica 'Ulisse Dini'\\ 
Universit\`{a} di Firenze \\ Viale Morgagni 67/A, \\  50134 Firenze , Italia}
\email{ottavian@math.unifi.it}
\urladdr{\tt{}}

\begin{abstract} 
We prove that the general tensor of size $2^n$ and rank $k$ has a unique decomposition as the sum of decomposable tensors if $k\le 0.9997\frac{2^n}{n+1}$ (the constant $1$ being the optimal value).
Similarly, the general tensor of size $3^n$ and rank $k$ has a unique decomposition as the sum of decomposable tensors if $k\le 0.998\frac{3^n}{2n+1}$ (the constant $1$ being the optimal value).

Some results of this flavor are obtained for tensors of any size, 
but the explicit bounds obtained are weaker.

\end{abstract}

\subjclass{}

 \maketitle

\thispagestyle{empty}

\section{Introduction}

We are interested in the problem to decompose a tensor in 
$\CC^{a_1+1}\otimes\ldots\otimes\CC^{a_q+1}$
as a sum of decomposable tensors. We study the decomposition of the general tensor
of given rank $k$. Conditions which guarantee the uniqueness of this decomposition 
are quite important in the applications \cite{KB}.
Indeed, many decomposition algorithms converge to {\it one} decomposition, so that a 
uniqueness result guarantees that the decomposition found is the one we looked for. 
Even from a purely theoretical point of view, the study of the decomposition shows 
some beautiful and not expected phenomena. After a look at the table
 in section \ref{gener}, we see that there are some exceptional sporadic 
 cases which are intriguing.

It is well known that when $k$ is bigger that the critical value
$$k_c:=\frac{\prod_{i=1}^q(a_i+1)}{1+\sum_{i=1}^q a_i}$$
then the decomposition can  never be unique (see the introduction of \cite{BC12}).
Only one case is known when $k_c$
is an integer and there is a unique decomposition for tensors of rank equal 
to $k_c$, namely when $q=3$ and {$a_1=1$, $a_2=a_3$.}

So, let us  consider the range $k<k_c$, where the problem can be understood better. 

Indeed, we consider two different cases, where the behavior is quite different. 
Let us order the numbers $a_i$, so that $a_1\le\ldots\le a_q$. 
The first case is when $a_q\ge \prod_{i=1}^{q-1}(a_i+1)-\left(\sum_{i=1}^{q-1}a_i\right)$,
that is when the last dimension is much bigger with respect to the others.
In this case there are always values of $k< k_c$ such that the general 
tensor of rank $k$ has a not unique decomposition. This case is completely 
described, {see Corollary \ref{corunbal}}.

The second remaining case is when $a_q\le \prod_{i=1}^{q-1}(a_i+1)-
\left(1+\sum_{i=1}^{q-1}a_i\right)$.
We believe in this case that the decomposition is almost always unique (for $k< k_c$), 
with the exception  of few cases which we list in section \ref{gener}
and have been studied in previous papers \cite{BC12, CO, CMO}. 
Further references can be found in \cite{CO}.
In general, we believe that the aforementioned list of exceptional cases is complete.

We illustrate some situations,  where our analysis cover almost all the possible ranks.

In the case of many copies of $\PP^1$, we prove (Theorem \ref{manyp1}) 
that the general tensor of rank $\le\frac{4095}{4096}k_c$ has a unique 
decomposition. Notice that $\frac{4095}{4096}=0.9997\ldots$, so we are very close 
to cover the range $k\le k_c$.

In the case of many copies of $\PP^2$, we prove (Theorem \ref{manyp2}) 
that the general tensor of rank $\le\frac{728}{729}k_c$ has a unique decomposition.
Notice that $\frac{728}{729}=0.998\ldots$. A similar result for many copies of $\PP^3$
is listed in the same section.

To give some evidence to our guess, that only a small set of exceptional cases occur, 
in the range $a_q\le \prod_{i=1}^{q-1}(a_i+1)-\left(1+\sum_{i=1}^{q-1}a_i\right)$,
we have implemented an algorithm that uses the concept of weak defectivity \cite{CC}.
With respect to the algorithm implemented in \cite{CO}, it is much faster, 
because it reduces the problem to numerical linear algebra operation,
while the algorithm in \cite{CO} needed Gr\"obner basis computations.

We can prove that the list of exceptional cases, appearing at the end of section \ref{gener} 
is complete for all $(a_1,\ldots, a_q)$ such that
$\prod_{i=1}^q(a_i+1)\le 100 $ (see Theorem \ref{listcomplete}).

Our general technique goes back to the seminal paper of Strassen \cite{Str},
and we owe a lot to his point of view. 
Strassen proved the case $c$ even of our Theorem \ref{stra}, which provides
a starting point for general results on cubic and general tensors (Theorem \ref{cub} and 
Theorem \ref{theorgen}).

In fact, we point out that our technique is inductive: once we know
that a particular Segre product $X$ of projective spaces satisfies 
the $k$-tangency condition in Lemma \ref{ml}, and consequently is $k$-identifiable 
for $k\leq \alpha k_c$, where $\alpha$ is some positive real $\leq 1$, then 
the same happens to be true for a larger product $X\times \PP^a$. 

This principle, which is indeed our main tool in the paper, is enlightened
in Corollary \ref{mlcor}. We hope that it will produce even more interesting
results, when applied to specific types of tensors that people working in
Multilinear Algebra are considering.

A consequence of this technique is a result on the dimension of secant varieties 
(Corollary \ref{corgen}), which can be seen as a generalization, to any 
number of factors, of previously known results in the case of three factors.
\smallskip
    
Let we finish with a short account of the status of the art, for
the identifiability of binary tensors, i.e. tensors in the span of
$\PP^1\times\dots\times\PP^1$. 
After the paper of Strassen \cite{Str}, 
and using methods of Algebraic Geometry, 
Elmore, Hall and Neeman proved in
\cite{EHN} the following asymptotic result: when the number $m$ of factors 
is ``very large" with respect to $k,a$, then the Segre product 
$\PP^a\times\dots\times\PP^a$ is $k$-identifiable.  
A much more precise  bound  for identifiability of binary products
was  obtained by Allman, Matias and Rhodes. In \cite{AMR} (Corollary 5)
they proved that the product of $m$ copies of $\PP^1$ is $k$-identifiable when
$m> 2\lceil\log_2 (k+1)\rceil +1$. Thus, they
gave a lower bound for $2^m$ which is quadratic with respect to $k+1$. 
Successively, using Geometric methods as well as a result by
Catalisano, Gimigliano and Geramita (\cite{CGG}),  
the first and second authors, in \cite{BC12} improved the bound, 
showing that a product of $m>5$ copies of $\PP^1$
is $k$-identifiable for all $k$ such that $k+1\leq 2^{m-1}/m$.
The case of $5$ copies of $\PP^1$ was shown to be exceptional.
The bound, which happened to be the best known up to now, 
is substantially improved in the present paper.

\section{Preliminaries}
We follow \cite{Land} for basic facts about the geometric point of view on tensors. 
For any irreducible projective varieties $X$, we denote by $S_k(X)$ the 
$k$-th secant variety of $X$, which is the Zariski closure of the set 
$\bigcup_{x_1,\ldots, x_k\in X}\langle x_1,\ldots, x_k\rangle$. 
In other words, $S_k(X)$ is the Zariski closure of the set of 
elements having $X$-rank equal to $k$.

We recall, from \cite{CO} Definition 2.1, the following:

\begin{defn} $X$ is called $k$-identifiable if the general element of  
$S_k(X)$ has a unique decomposition as the sum of $k$ elements of $X$.
\end{defn}

Those who wonder why we ask for the uniqueness just for a {\it general}
element of $S_k(X)$, should consider that {for any $k\ge 2$} there are
always points of $S_k(X)$ {which have rank smaller than $k$.}

Notice that $k$-identifiable implies $(k-1)$-identifiable, and so on. 

A fundamental Geometric tool for the analysis of the identifiability of tensors,
is Proposition 2.4 in \cite{CO}, which is essentially a consequence of Terracini Lemma.

\begin{prop} \label{wd-id}
If there exists a set of $k$ particular points 
$x_1,\ldots x_k\in X$, such  that the span 
$\langle{\mathbb T}_{x_1}X,\ldots, {\mathbb T}_{x_k}X\rangle$ contains 
${\mathbb T}_xX$ only if $x=x_i$ 
for some $i=1,\ldots k$, then $X$ is $k$-identifiable.
\end{prop}

\section{The Main Lemma}\label{mainlem}

The inductive step, that allows us to provide effective results on the identifiability
of tensors, relies in the following:

\begin{lem}\label{ml} Let $X$ be a smooth non-degenerate projective subvariety 
of $\PP^N$, of dimension $n$. Let $Y$ denote the canonical Segre
embedding of $X\times\PP^m$ into $\PP^M$, $M=mN+m+N$.
Fix $k$ with $(n+1)k<N+1$ and $r<N$ such that
$r+1\geq (n+m+1)k$. 
Assume that a general linear 
subspace of $\PP^N$, of dimension $r$, which is 
tangent to $X$ at $k$ general points, is not tangent to $X$ elsewhere.

Then the general linear subspace of $\PP^M$, of dimension 
$mr+m+r$, which is tangent to $Y$ at $(m+1)k$ general points,
is not tangent to $Y$ elsewhere. 
\end{lem}
\begin{proof}
First of all, notice that $\dim(Y)=(m+n)$, and 
$(m+1)(r+1)\geq (m+n+1)(m+1)k$. 
Thus, by an obvious parameter count, there are linear subspaces 
of dimension $mr+m+r$ which
are tangent to $Y$ at $(m+1)k$ general points.

Fix $m+1$ independent points $p_0,\dots,p_m$ of $\PP^m$
and for $j=0,\dots,m$ take $k$ general points $q_{ij}$
of the fiber $X\times\{p_j\}$. Call $\pi_j$ the natural
projection of  $X\times\{p_j\}$ to $X$.

For $h=0,\dots,m$,
fix a general linear subspace $R_h$, of dimension $r$,
which is tangent to $X\times\{p_h\}$ at the $k$ points 
$q_{1h},\dots,q_{kh}$ and passes through the points $\pi_j(q_{ij})\times \{p_h\}$,
for $j\neq h$. Since $r+1\geq k(n+1)+km$, such spaces $R_h$ exist.
Moreover $R_h$ it is tangent to $X\times\{p_h\}$ only
at the point $q_{1h},\dots, q_{kh}$, by our assumption on $X$.

Let $R$ be the span of all the $R_h$'s. We claim
that $R$, which is a linear subspace of dimension $mr+m+r$,
is tangent to $Y$ at all the points $q_{ij}$, and it is not
tangent to $Y$ elsewhere. This will conclude the proof
of the lemma, by semicontinuity.

First notice that for all $i,j$, $R$ contains $m+1$ general points
of $\{\pi_j(q_{ij})\}\times\PP^m$, hence it contains these fibers.
Since $R$ also contains the tangent spaces to $X\times\{p_h\}$
at the points $q_{ih}$'s for all $h$, then it is tangent 
to $Y$ at all the points $q_{ij}$'s.

Assume now that there exists a point $x\in Y$, different from
the  $q_{ih}$'s, such that $R$ is tangent to $Y$ at $x$.
Call $x'$ the projection of $x$ to $\PP^m$, so that in 
some coordinate system, we can write $x'= a_0p_0+\dots+a_mp_m$.
There is at least one of the $a_i$'s, say $a_0$, which is 
non-zero. Assume that also $a_1\neq 0$. Then, the projection
of $R$ to  $\PP^N\times\{p_0\}$, which by construction coincides
with $R_0$, is also tangent to $X\times\{p_0\}$
at the projection of $q_{k1}$. By the generality of the 
choice of the $q_{ij}$'s, $q_{k1}$ cannot coincide with any of
the points $q_{10},\dots,q_{k0}$. Thus we get a contradiction.

So, we conclude that $a_1=0$. Similarly we get that $a_2=\dots=a_m=0$.
It follows that $x=x'$ belongs to $X\times\{p_0\}$ and 
since $R_0$ is tangent to $X\times\{p_0\}$ at $x$, then
$x$ must coincide with some point $q_{i0}$. 
\end{proof}

\begin{rem}\label{menok}
It is worthy of spending one Remark to point out that, by semicontinuity,
if a general linear subspace of $\PP^N$, of dimension $r$, which is 
tangent to $X$ at $k$ general points, is not tangent to $X$ elsewhere,
then the same phenomenon occurs for general linear subspaces of dimension
$r-1$, $r-2$, and so on.
\end{rem}

The Lemma, together with Theorem \ref{wd-id},
 produces the following general principle:
 
\begin{cor} \label{mlcor} 
With the same assumptions on $X$ of Lemma \ref{ml},
then $Y=X\times\PP^m$ is {$(m+1)k$-identifiable}. 
\end{cor}

Thus we will prove the identifiability of Segre products,
starting with a $X$ who is a Segre product for which we know 
that the assumptions of Lemma \ref{ml} hold (by 
computer-aided specific computations or by Theorem \ref{stra} below)
and then extending the number of factors of $X$,
 and using {Lemma \ref{ml}}
inductively.

\section{Many copies of $\PP^1$}\label{piunos}

The main case in which the previous result applies is 
the Segre product of many projective lines.

\begin{prop} Let $X$ be the product of $n$ copies of $\PP^1$,
$6\le n\le 12$, naturally embedded in $\PP^{2^n-1}$. Then for 
$k<k_c=\frac{2^n}{n+1}$ the linear span of $k$ general tangent spaces at $X$,
is not tangent to $X$ elsewhere. In particular $X$ is $k$-identifiable 
for all $k<k_c$.
\end{prop}
\begin{proof} Just a computer-aided computation, following 
the the algorithm presented in section \ref{algo}. In the case of $12$ copies,
the algorithm goes out of memory if implemented in a straightforward way. We
used a ``divide and conquer'' technique to save memory, running in 2 hours on
a PC with two processors at 2GHz.
\end{proof}

\begin{thm} For $n\geq 12$, let $X$ be the product of 
$n$ copies of $\PP^1$, naturally embedded in $\PP^N$,
with $N=2^n-1$. Then for $r< 2^n-2^{n-12}$ and for
{$k\leq\frac{r+1}{n+1}$}, a general linear
subspace of dimension $r$ which is tangent to $X$
at $k$ points, is not tangent to $X$ elsewhere. 
\end{thm}
\begin{proof} The proof goes by induction on $n\geq 12$.
If $n=12$, the claim follows from the previous Proposition
and Remark \ref{menok}.

Assume the claim holds for $n-1$. Again by 
Remark \ref{menok}, it suffices to prove the claim
for $r+1=2^n-2^{n-12}$. Fix $k$ as above. Notice that
$k':=\lceil k/2\rceil$ is at most $(2^{n-1}-2^{n-13})/(n+1)+1$,
which is smaller than $(2^{n-1}-2^{n-13})/n$ for $n\geq 12$.
Thus we may apply induction: the general linear subspace of
$P^{N'}$, $N':=2^{n-1}-1$, of dimension $(2^{n-1}-2^{n-13})-1$,
which is tangent to $X':=(\PP^1)^{n-1}$ at $k'$ points, is not
tangent to $X'$ elsewhere. The claim now follows directly
from the Main Lemma \ref{ml}.
\end{proof}

\begin{rem} The assumption $n\ge 6$ is motivated by the fact that
for $5$ copies of $\PP^1$, $X$ is $k$-identifiable if and only if 
$k\le 4$, while the general tensor
of rank $5$ has exactly two decompositions \cite{BC12}.

We recall that for $4$ copies of $\PP^1$, 
$X$ is $k$-identifiable if and only if 
$k\le 2$, while it is a result of Strassen that the general tensor
of rank $3$ has infinitely many decompositions.

For $3$ copies of $\PP^1$, 
$X$ is $k$-identifiable if and only if 
$k\le 2$, $2$ being the general rank.
\end{rem}

As a corollary, we get

\begin{thm} \label{manyp1}For $n\geq 12$, let $X$ be the product of 
$n$ copies of $\PP^1$, naturally embedded in $\PP^N$,
with $N=2^n-1$. Then for 
$k\leq (2^n-2^{n-12})/(n+1)$, $X$ is $k$-identifiable.
\end{thm}

Let us compare the result with the (best known) bound 
on the identifiability of $\PP^1\times\dots\times\PP^1$
given in \cite{BC12}, which makes a fundamental use of the
main result in \cite{CGG}.

The main result of \cite{BC12} proves that $X$ is $k$-identifiable
for $k< 2^{n-1}/n$, which is a little better than half way from 
the critical (maximal) value $k_c$.

The previous result shows that $X$ is $k$-identifiable for 
$$k\leq \frac{2^n}{n+1}(1-\frac 1{2^{12}}) = \frac{4095}{4096}\ 
\frac{2^n}{n+1} = {0,9997\ldots} \cdot k_c $$
a sensible improvement, as $n$ grows.

\section{Many copies of $\PP^2$ and $\PP^3$}\label{pidues}

Let us see what happens with
the Segre product of many projective planes.

\begin{prop} Let $X$ be the product of $n$ copies of $\PP^2$,
$4\le n\le 6$, naturally embedded in $\PP^{3^n-1}$. Then for 
$k<k_c=\frac{2^n}{n+1}$ the linear span of $k$ general tangent spaces at $X$,
is not tangent to $X$ elsewhere. In particular $X$ is $k$-identifiable 
for all $k<k_c$.
\end{prop}
\begin{proof} Just a computer-aided computation, following 
the the algorithm presented in section \ref{algo}.
\end{proof}

\begin{rem}
The assumption $n\ge 4$ is motivated by the fact that
for $3$ copies of $\PP^2$, $X$ is $k$-identifiable if and only if 
$k\le 3$, while it is a result of Strassen {(\cite{Str} \S 4)} that the general tensor
of rank $4$ has infinitely many decompositions.
\end{rem}

\begin{thm} \label{manyp2}For $n\geq 6$, let $X$ be the product of 
$n$ copies of $\PP^2$, naturally embedded in $\PP^N$,
with $N=3^n-1$. Then for $r< 3^n-3^{n-6}$ and for
$k\leq (r+1)/(2n+1)$, a general linear
subspace of dimension $r$ which is tangent to $X$
at $k$ points, is not tangent to $X$ elsewhere. 
\end{thm}
\begin{proof} The proof goes by induction on $n\geq 6$.
If $n=6$, the claim follows from the previous proposition and
Remark \ref{menok}.

Assume $n\geq 7$ and the claim holds for $n-1$. Again by 
Remark \ref{menok}, it suffices to prove the claim
for $r+1=3^n-3^{n-6}$. Fix $k$ as above. Notice that
$k':=\lceil k/3\rceil$ is at most $(3^{n-1}-3^{n-7})/(2n+1)+1$,
which is smaller than $(3^{n-1}-3^{n-7})/(2n-1)$ for $n\geq 5$.
Thus we may apply induction: the general linear subspace of
$P^{N'}$, $N':=3^{n-1}-1$, of dimension $(3^{n-1}-3^{n-7})-1$,
which is tangent to $X':=(\PP^2)^{n-1}$ at $k'$ points, is not
tangent to $X'$ elsewhere. The claim now follows directly
from the Main Lemma \ref{ml}.
\end{proof}

As a corollary, we get

\begin{thm} For $n\geq 6$, let $X$ be the product of 
$n$ copies of $\PP^2$, naturally embedded in $\PP^N$,
with $N=3^n-1$. Then for 
$k\leq (3^n-3^{n-6})/(2n+1)$, $X$ is $k$-identifiable.
\end{thm}

The previous result shows that $X$ is $k$-identifiable for 
$$k\leq  \frac{3^n}{2n+1}(1-\frac 1{3^6}) = \frac{728}{729}\ \frac{3^n}{2n+1},$$
i.e. up to $728/729=0.998\ldots$ of the critical (maximal) value $k_c$.
\smallskip

And now the reader can see how the trick goes, at least for cubic tensors.
Once one determines a starting point,  for few copies of given projective spaces
(e.g. by using a computer-aided computation),
then the Main Lemma \ref{ml} provides an extension to the
product of an arbitrary number of copies of projective spaces, in which
the bound is expressed as a constant fraction of the critical value $k_c$.
\smallskip

We end the list of particular cases with the product of many copies
of $\PP^3$, which is relevant because of its connection with the 
Algebraic Statistics of DNA chains.

 \begin{thm} \label{manyp3} 
Let $X$ be the product of 
$n\geq 5$ copies of $\PP^3$, naturally embedded in $\PP^N$,
with $N=4^n-1$. 

(i) for $n=5$,  a general linear subspace of dimension $r=1007$ 
 which is tangent to $X$
at $k\leq 63$ points, is not tangent to $X$ elsewhere. 

(ii) For $n>5$ and $k\leq (4^n-4^{n-3})/(3n+1)$,
a general linear subspace of dimension $r=4^n-4^{n-3}-1$ 
 which is tangent to $X$ at $k$ points, is not tangent to $X$ elsewhere. 
 
(iii) For  $k\leq (4^n-4^{n-3})/(3n+1)$, then  $X$ is $k$-identifiable.
 In other words, $X$ is $k$-identifiable up to 
$63/64={0.98\ldots}$ of the critical (maximal) value $k_c$.
\end{thm}
\begin{proof} (i) follows from a computer-aided computation, following 
the the algorithm presented in section \ref{algo}.
(ii) is a consequence of (i) and the inductive Lemma \ref{ml}.
(iii) follows from (ii) and Theorem \ref{wd-id}.
\end{proof}

\section{Products of three projective spaces}\label{tre}

For the general case, in which we have projective spaces of arbitrary dimension,
in order to produce examples 
similar to the ones of the previous section,
we need a starting point for the induction.

We obtain a starting point, for the case of the product of three
projective spaces $X=\PP^a\times\PP^b\times\PP^c$, $2<a\leq b \leq c$,
 from the following Theorem, 
which is due to Strassen in the case
$c$ odd (see \cite{Str}, Corollary 3.7), 
and we generalize to any $c$.

Our proof is apparently independent from the argument given by Strassen.
Indeed, following correctly the details of the steps, 
one realizes that the two arguments are essentially equivalent.

\begin{thm}\label{stra} Let $X$ be the product of three projective spaces  
$X=\PP^a\times\PP^b\times\PP^c$, $2<a\leq b \leq c$,
naturally embedded in $\PP^N$, with $N=(a+1)(b+1)(c+1)-1$. Then 
a general linear subspace $L$ of codimension $a+b+2$ in $\PP^N$, that
contains the span 
of the tangent spaces to $X$ at $k$ general points, with:
$$k{\le}\frac {(a+1)(b+1)(c+1)}{a+b+c+1}-c-1,$$
is not tangent to $X$ elsewhere.
\end{thm}
\begin{proof}
Let $\PP^c=\PP(C)$, where $C$ is a vector space of dimension $c+1$.
Fix one vector $v_0\in C$ and split $C$ in a direct sum
$C= \langle v_0\rangle \oplus C'$, where $C'$ is a supplementary subspace
of dimension $c$. From the geometric point of view, this is equivalent
to split the product $X$ in two products
$$X'=\PP^a\times\PP^b\times\PP^{c-1} \mbox{ and } X''=\PP^a\times\PP^b\times\{P_0\}
=\PP^a\times\PP^b.$$
Fix general points $P_1,\dots, P_k\in X'$, with $P_i=v_i\otimes
w_i\otimes u_i$ and let $Q_1,\dots, Q_k$, $Q_i=v_i\otimes w_i$,
be the corresponding points of $X''$. The linear span  of
the $Q_i$'s is a space of dimension $k-1$ in $\PP^{N''}$,
where $N'' = ab+a+b$. 

By assumption $k-1{\le} N''-\dim(X'')=N''-a-b$.
Indeed if $c+1\geq a+b$ then
$$ k-1\leq \frac {(a+1)(b+1)(c+1)}{a+b+c+1}-a-b-1\leq (a+1)(b+1)-a-b-1.$$
If $c+1<a+b$ then $k< (a+1)(b+1)/2$ and $(a+1)(b+1)/2>a+b$.

Fix a linear space $L''$ of codimension $a+b+1$ in $\PP^{N''}$, which
contains the span of the $Q_i$'s.
Since the points $Q_i$'s are general in $X''$, it follows from
the Theorem 2.6 in \cite{CC1} (it is a generalization of the ``trisecant lemma'')
that the linear space $L''$ does not meet $X''$ in other points.
Moreover $L''$ is not tangent to $X''$ at any of the points $Q_i$'s.

Let $L'$ be a hyperplane in $\PP^{N'}$, $N'=(a+1)(b+1)c-1$, which is tangent
to  $X'$ at the points $P_i$'s. The hyperplane $L'$ exists, since
by assumption 
$$ k(\dim(X')+1) < (a+1)(b+1)(c+1)-c(a+b+c)< N-1.$$

Let $L$ be the linear span of $L'$ and $L''$.
$L$ has codimension $a+b+2$ and it is tangent to $X$ at the 
$k$ points $P_1, \dots, P_k$,
since it contains the tangent spaces to $X'$ at the $P_i$'s,
moreover it contains the points $Q_i$'s, so it contains the
fiber $\PP^c$ passing through each $P_i$. 

We want to exclude that $L$ is tangent to $X$ at any other point $P\neq P_i$.
Call $Q$ the projection of $P$ to $X''$. If $L$ is tangent to $X$ at $P$,
then it must contain the fiber $\PP^c$ passing through $P$,
thus it contains $Q$. This proves that $Q$ is one of the $Q_i$'s
(say $Q=Q_1$), since $L$ does not meet $X''$ elsewhere.
But then $L$ contains the fibers $\PP^a$ and $\PP^b$ at two points 
$P,P_1$ with the same projection to $X''$.
Thus it contains these fibers at any point of the line $\ell$
joining $P,P_1$. As $\ell$ contains $Q_1$, we get a contradiction, since
$L''=L\cap \PP^{N''}$ is not tangent to $X''$ at $Q_1$.
\end{proof}

\begin{cor}\label{trex} Let $X$ be the product of three projective spaces  
$X=\PP^a\times\PP^b\times\PP^c$, $2<a\leq b \leq c$,
naturally embedded in $\PP^N$, with $N=(a+1)(b+1)(c+1)-1$. Then for
$$k{\le}\frac {(a+1)(b+1)(c+1)}{a+b+c+1}-c-1,$$
$X$ is $k$-identifiable.
\end{cor}
\begin{proof} Follows immediately from the previous Theorem and \cite{CO}.
\end{proof}

The identifiability of products of three projective spaces has been studied
by a long list of authors, who refined the celebrated Kruskal's bound
for arbitrary tensors. 	We mention De Lauthawer's results for unbalanced
tensor (\cite{DeL}), and the general bounds found by the second and third
authors in \cite{CO}.

We believe that the bound of Corollary \ref{trex}, at least for some balanced case,
is  the best known result for tensors of type $a,b,c$.

\section{Inductive bounds for the identifiability of general tensors}\label{gener}

The same procedure we used for products of many 
projective lines and planes, based on the bound found in Corollary
\ref{trex}, can produce results 
for {\it cubic} tensors, which, in some cases, are far beyond any
known result on the identifiability problem.

Then, with the above notation, we have:

\begin{thm}\label{cub} For $n\geq 3$, let $X$ be the product of 
$n$ copies of $\PP^a$, naturally embedded in $\PP^N$,
with $N=(a+1)^n-1$. Then for $r<(a+1)^n-(3a+1)(a+1)^{n-2} $ and for
$k\leq (r+1)/(an+1)$, a general linear
subspace of dimension $r$ which is tangent to $X$
at $k$ points, is not tangent to $X$ elsewhere. 

As a consequence, we get that $X$ is $k$-identifiable, for 
$$k\leq \frac{(a+1)^n-(3a+1)(a+1)^{n-2}}{an+1}.$$ 
\end{thm}
\begin{proof} The proof is absolutely similar to the ones 
of the cases {$a=1,2 ,3$ given above. We may assume $a\ge 4$.} It goes by induction on $n\geq 3$,
and uses { Theorem \ref{stra}} as a starting point.

We leave the straightforward details to the reader.
\end{proof}

We recall that we defined, in the introduction, the critical value
$$k_c=  \frac{\prod_{i=1}^q(a_i+1)}{1+\sum_{i=1}^q a_i}$$
which is essentially the maximum for which $k$-identifiability 
can hold. Then the previous bound proves that $X$ is $k$-identifiable, for 
$$k{\le} \frac{a(a-1)}{(a+1)^2}\ k_c.$$

Even for the case of rectangular tensors, we are able to prove 
{some} results, using the same procedure.

\begin{thm}\label{theorgen}
Let $X$ be the product of $q\geq 3$ projective spaces
$X=\PP^{a_1}\times\cdots\times\PP^{a_q}$, naturally embedded in $\PP^N$,
with {$N=-1+ \Pi_{i=1}^{q} (a_i+1)$. Then for 

$r<\Pi_{i=1}^{q}(a_i+1)-(a_1+a_2+a_3+1)\Pi_{i=3}^{q}(a_i+1)$ and for
$k\leq (r+1)/(1+\sum_{i=1}^{q} a_i)$, a general linear
subspace of dimension $r$ which is tangent to $X$
at $k$ points, is not tangent to $X$ elsewhere. 

As a consequence, we get that $X$ is $k$-identifiable, for 
$$k\leq \frac{\Pi_{i=1}^{q} (a_i+1)-(a_1+a_2+a_3+1)\Pi_{i=3}^{q}(a_i+1)}{1+\sum_{i=1}^q a_i} .$$}
\end{thm}

Again, notice that we get $k$-identifiability for
$$k{\le}\frac
{a_1a_2-a_3}{(a_1+1)(a_2+1)}\ k_c.$$

Of course, the previous bound changes if one reorders the $a_i$ suitably.
Notice that the previous thereom requires $a_1a_2>a_3$ in order to give a an 
effective range of values for $k$. Moreover, one of the conditions among  
$a_1\gg a_3$, $a_2\gg a_3$ and $a_1a_2\gg a_3$  is strongly preferable 
to have a larger range of values for $k$.
\smallskip

We strongly believe that some {\it ad hoc} procedure, as well as
the improvements of our computational facilities, for the starting
point of the induction, are suitable to produce advancement
in the inequalities of our results.

Let us stress that the previous bounds provide also some answers 
to the problem of finding the dimension of secant varieties to Segre
varieties (i.e. to the dimension of paces of tensors of given rank).

\begin{cor} \label{corgen} Let $X$ be the product of $q\geq 3$ projective spaces
$X=\PP^{a_1}\times\cdots\times\PP^{a_q}$.
If $$k\leq \frac{\Pi_{i=1}^{q} (a_i+1)-(a_1+a_2+a_3+1)\Pi_{i=3}^{q}(a_i+1)}{1+\sum_{i=1}^{q} a_i}$$
then the dimension of the $k$-secant variety $S_k(X)$ is the expected one,
namely it is equal to  $k(1+\sum_{i=1}^q a_i)-1$.
\end{cor}

\begin{rem}
One should compare the previous result with the result of Gesmundo, \cite{Ges}, Theor. 1.1,
who proved that the dimension of $k$-secant variety to 
$\PP^{a_1}\times\ldots\times\PP^{a_q}$
is the expected one for $k\le \Theta_q k_c$, where $\Theta_q$ is a 
constant depending only on $q$. Moreover, in the case where $a_i+1$ 
are powers of $2$, then $\Theta_q\to 1$ when $q\to\infty$.
\end{rem}

Next, we show a list of Segre products for which $k$-identifiability
does not hold.  We refer to  section 5 of \cite{CO} 
for further details.
\smallskip

{\bf Table of known cases when $a_q\le \prod_{i=1}^{q-1}(a_i+1)-
\left(1+\sum_{i=1}^{q-1}a_i\right)$ and the decomposition of the general tensor
of rank $k<k_c$ is not unique.}

{\small
$$\begin{array}{ccl}
(a_1,\ldots,a_q)&k& \textrm{number of decompositions}\\
\hline\\
(2,3,3)&5&\infty^1\\
\hline\\
(2,b,b)\quad b\textrm{\ even}&
\frac{3b+2}{2}&\infty^{\frac b2+1}\\
\hline\\
(1,1,n,n)&2n+1&\infty^1\\
\hline\\
(3,3,3)&6&2\\ 
\hline\\ 
(2,5,5)&8& \textrm{finite}, \ge 6\\
\hline\\
(1,1,1,1,1)&5&2\\
\hline\\
\end{array}$$
}
A straightforward application of the algorithm presented in the last section shows the following

\begin{thm}\label{listcomplete}The previous list is complete for all $(a_1,\ldots, a_q)$ such that
$\prod_{i=1}^q(a_i+1)\le 100 $.
\end{thm}

\section{The unbalanced case}

Consider again $X=\PP^{a_1}\times\cdots\times\PP^{a_q}$
be a Segre product, canonically embedded in $\PP^N$, where
$N+1=\prod_{i=1}^q (a_i+1)$. We may assume $a_1\le\ldots\le a_q$ {and $q\ge 3$}.

When the dimension of $a_q$ is much bigger than the dimension of the other factors, then the tensor decomposition has a not expected behavior, that can be understood by considering the Segre variety consisting of only two factors,
$\PP^{a_q}$ and the product of all the others.

In \cite{AOP} it was settled completely the dimension of secant variety in the unbalanced case
$a_q\ge \prod_{i=1}^{q-1}(a_i+1)-\left(\sum_{i=1}^{q-1}a_i\right)+1$.
In these cases the last secant variety, which does not fill the ambient space, 
always has dimension smaller than expected.

{We consider now the $k$-identifiability. It turns out (see Corollary \ref{corunbal})
that the analogous condition to be unbalanced is a bit weaker, namely 
$a_q\ge \prod_{i=1}^{q-1}(a_i+1)-\left(\sum_{i=1}^{q-1}a_i\right)$. This explains why in the table
shown in \S \ref{gener} we considered just the remaining cases.

The following Proposition is certainly well known, we prove it for the convenience of the reader.

\begin{prop}\label{rankab} The general tensor of rank {{$k\le \prod_{i=1}^{q-1}(a_i+1)-\left(1+\sum_{i=1}^{q-1}a_i\right)$} in
$\PP(\CC^{a_1+1}\otimes\ldots\otimes\CC^{a_q+1})$
has a unique decomposition as sum of $k$} decomposable summands for $a_q\ge \prod_{i=1}^{q-1}(a_i+1)-\left(2+\sum_{i=1}^{q-1}a_i\right)$.
\end{prop}
\begin{proof}
Let $\phi\in\CC^{a_1+1}\otimes\ldots\otimes\CC^{a_q+1}$ be general of rank {$k\le \prod_{i=1}^{q-1}(a_i+1)-\left(1+\sum_{i=1}^{q-1}a_i\right)$}.
It induces the flattening  contraction operator 
$$A_{\phi}\colon(\CC^{a_q+1})^{\vee}\to\CC^{a_1+1}\otimes\ldots\otimes\CC^{a_{q-1}+1}$$
which has still rank {$k$}, by the assumption $a_q\ge \prod_{i=1}^{q-1}(a_i+1)-\left(2+\sum_{i=1}^{q-1}a_i\right)$. 
Indeed, if {$\phi=\sum_{i=1}^{k}v_{i,1}\otimes v_{i,2} \otimes \cdots \otimes v _{i,q}$}
with $v_{i,j}\in \CC^{a_j+1}$,
where $v_{i,q}$ can be chosen as part of a basis of $\CC^{a_q+1}$,
 then $\textrm{Im\ }A_{\phi}$ is the span of the representatives
of $v_{i,1}\otimes\ldots\otimes v_{i,q-1}$ for {$i=1,\ldots , k$}.
It is well known that the projectification of this span,
whose dimension is smaller than the codimension of the Segre variety
$Y=\PP^{a_1}\times\ldots\times\PP^{a_{q-1}}\subset\PP(\CC^{a_1+1}\otimes\ldots\otimes\CC^{a_{q-1}+1})$, meets $Y$
only in these {$k$} points
(see again, for example, the Theorem 2.6 in \cite{CC1}).
The claim follows.
\end{proof}

{The case $q=3$ of next Propositions \ref{dif0} and \ref{different} is contained in Prop. 5.4 of \cite{CO}.}
\begin{prop}\label{dif0}
If $a_q\ge \prod_{i=1}^{q-1}(a_i+1)-\left(\sum_{i=1}^{q-1}a_i\right)+1$ 
then the rank of a {general} tensor in 
 $\PP(\CC^{a_1+1}\otimes\ldots\otimes\CC^{a_q+1})$  is $\min\{a_q+1, 
 \prod_{i=1}^{q-1}(a_i+1)\}$. 

Moreover $\PP^{a_1}\times\ldots\times\PP^{a_q}$ is not $k$-identifiable for
$k>\prod_{i=1}^{q-1}(a_i+1)-\left(\sum_{i=1}^{q-1}a_i\right)$.
If $a_q= \prod_{i=1}^{q-1}(a_i+1)-\left(\sum_{i=1}^{q-1}a_i\right)$, 
then the rank of a {general} tensor in 

 $\PP(\CC^{a_1+1}\otimes\ldots\otimes\CC^{a_q+1})$  is $a_q+1$.
\end{prop}
\begin{proof}
When $a_q\ge \prod_{i=1}^{q-1}(a_i+1)-\left(\sum_{i=1}^{q-1}a_i\right)+1$, 
we are in the unbalanced case, 
according to Definition 4.2 of \cite{AOP} (in the defective setting,
the range of the unbalanced case is slightly bigger than in the weakly 
defective setting). In this case the statement follows from
 Theorem 4.4  of \cite{AOP}. 

When $a_q= \prod_{i=1}^{q-1}(a_i+1)-\left(\sum_{i=1}^{q-1}a_i\right)$, 
using the same technique, we show that the 
secant variety $S_k(\PP^{a_1}\times\ldots\times\PP^{a_q})$ has the expected dimension, 
for $k\le a_q$, and fills the  ambient space, for $k=a_q+1$. 

Indeed, with the notations of \cite{AOP}, condition
$T(a_1,\ldots, a_q;a_q;0^q)$ reduces to {$T(a_1,\ldots,a_{q-1},0;1;0^{q-1},a_q-1)$ and
$T(a_1,\ldots,a_{q-1},0;0;0^{q-1},a_q)$}
which are true and subabundant, while condition $T(a_1,\ldots, a_q;a_q+1;0,0,0)$ reduces 
to condition {$T(a_1,\ldots, a_{q-1},0;1;0^{q-1},a_q)$}  which is superabundant and true.
\end{proof}

\begin{prop}\label{different}
Assume $a_q\ge \prod_{i=1}^{q-1}(a_i+1)-\left(\sum_{i=1}^{q-1}a_i\right)$. 
Then the number of different decompositions of a general 
tensor of rank $k=\prod_{i=1}^{q-1}(a_i+1)-\left(\sum_{i=1}^{q-1}a_i\right)$ is ${D\choose k}$
where 
$D=\deg \PP^{a_1}\times\ldots\times\PP^{a_{q-1}}=\frac{(\sum_{i=1}^{q-1}a_i)!}
{a_1!\ldots a_{q-1}!}$. 
{This number is always bigger than $1$, so we have never identifiability.}
\end{prop}
\begin{proof}
We apply the same argument of the proof of Proposition \ref{rankab}.
We pick a general $\phi$ of rank $k$. The only difference is that, now, the dimension of the
projectification of $\textrm{Im\ }A_{\phi}$, which is $k-1$, equals the codimension
of $\PP^{a_1}\times\ldots\times\PP^{a_{q-1}}$. Thus we get $D$ points of intersection.
Any choice of {$k$} among these $D$ points 
yields a decomposition.
\end{proof}

\begin{cor}\label{corunbal}
Assume $a_q\ge \prod_{i=1}^{q-1}(a_i+1)-\left(\sum_{i=1}^{q-1}a_i\right)$. 
Then $\PP^{a_1}\times\ldots\times\PP^{a_{q}}$ is $k$-identifiable if and only if
$$k\le \prod_{i=1}^{q-1}(a_i+1)-\left(1+\sum_{i=1}^{q-1}a_i\right)$$
\end{cor}
}

\begin{rem}
We notice a misprint in the table in section 5 of \cite{CO}.
The condition for ``defective unbalanced'' should read as
$c\ge (a-1)(b-1)+3$ instead of $c\ge (a-1)(b-1)+1$.
The proof of Prop. 5.3 of \cite{CO} needs slight modifications accordingly, 
but the statement remains correct.
\end{rem}

\section{The algorithm}\label{algo}
The algorithm we have used has been implemented in Macaulay2 \cite{GS} and it 
can be found as ancillary file in the arXiv submission of this paper.

The steps are the following.

\begin{enumerate}
\item{} We choose $s$ random points $p_1,\ldots, p_s$ on the Segre variety $X$, 
working on an affine chart. The point $p_1$ can be chosen as $(1,0,\ldots)$
on each factor.

\item{} We compute the  equations of the span of tangent spaces
 $<T_{p_1},\ldots, T_{p_s}>$.

\item{} For any of the cartesian equations we compute its partial derivatives,
the common locus is the locus $C$ of points $p$ such that 
$T_pX\subset \langle T_{p_1},\ldots, T_{p_s}\rangle$.

\item{} We compute the rank of the jacobian matrix of $C$ at $p_1$.
If it is equal to the dimension of $X$ then $X$ is $k$-identifiable.
If it is smaller than the dimension of $X$ then a further analysis is required.
\end{enumerate}


\begin{thebibliography}{KMR98}
\bibitem[AOP09]{AOP} H. Abo, G. Ottaviani, C. Peterson, 
\newblock\emph{Induction for secant varieties of Segre varieties}.
\newblock{Trans. Amer. Math. Soc.}, 361 (2009), 767--792

\bibitem[AMR09]{AMR} E.S. Allman, C. Matias, J.A. Rhodes.
\newblock\emph{Identifiability of parameters in latent structure models 
with many observed variables}.
\newblock{Ann. Statist.}, 37  (2009), 3099-3132.

\bibitem[BC12] {BC12}
C. Bocci, L. Chiantini 
\newblock \emph{On the identifiability of binary Segre products.}
\newblock{J. Alg. Geom.}22 (2013) 1--11. 

\bibitem[CGG11] {CGG}
M.V. Catalisano, A.V. Geramita, A. Gimigliano
\newblock \emph{Secant varieties of $\PP^1\times\ldots\times\PP^1$ (n-times) 
are not defective for $n\ge 5$}.
 \newblock{J. Algebraic Geom.} 20 (2011), no. 2, 295–-327.

\bibitem[CC02]{CC}
L. Chiantini and C. Ciliberto,
\newblock \emph{Weakly defective varieties.}
\newblock {\em Trans. Amer. Math. Soc.}, 354(1) (2002) 151--178. 


\bibitem[CC06] {CC1}
L. Chiantini and C. Ciliberto.
\newblock \emph{ On the concept of $k$-secant order of a variety}.
\newblock{ J. London Math. Soc.} 73 (2006), 436--454.


\bibitem[CO12] {CO}
L. Chiantini, G. Ottaviani,
\newblock\emph{On generic identifiability of 3-tensors of small rank}.
\newblock{SIAM J. Matrix Anal. Appl.}, 33 (3), (2012), 1018--1037. 

\bibitem[CMO13]{CMO} L. Chiantini, M. Mella, G. Ottaviani,
\emph{One example of general unidentifiable tensors}, arXiv:1303.6914


\bibitem[EHN05]{EHN} R. Elmore, P. Hall, A. Neeman.,
\newblock\emph{An application of classical invariant theory to identifiability in
non-parametric mixtures}. 
\newblock{ Ann. Inst. Fourier}, 55  (2005), 1--28.


\bibitem[Ges12] {Ges}
F. Gesmundo, 
\newblock\emph{An asymptotic bound for secant varieties of Segre varieties.}
to appear in Annali dell'Universit\`a di Ferrara, arXiv:1209.1732

\bibitem[GS]{GS} D.~Grayson and M.~Stillman,
\newblock Macaulay 2, a software system for research in algebraic geometry.
\newblock Available at {\tt www.math.uiuc.edu/Macaulay2/ }.


\bibitem[KB09]{KB}
T.~Kolda and B.~Bader. 
\newblock \emph{Tensor Decompositions and Applications.} 
\newblock{ SIAM Review}, 51(3) (2009) 455--500.

\bibitem[Land]{Land} J.~M.~Landsberg. 
\newblock The geometry of tensors with applications.
\newblock {\em Graduate Studies in Mathematics} 128, AMS, Providence 2012.

\bibitem[Lat06]{DeL} L. De Lauthawer. 
\newblock \emph{A link between the canonical decomposition in multilinear algebra and
simultaneous matrix diagonalization.}
\newblock{SIAM J. Matrix Anal. Appl.}, 28 (2006) 642--666.

\bibitem[Str83] {Str}
V. Strassen
\newblock \emph{Rank and optimal computation of generic tensors.} 
\newblock {Linear Algebra Appl.}, 52 (1983) 645--685. 



\end{thebibliography}
\end{document}